\documentclass[english]{amsart}
\usepackage{ae,aecompl}
\usepackage[T1]{fontenc}
\usepackage[latin1]{inputenc}
\usepackage{amsthm}
\usepackage{amstext}
\usepackage{amssymb}
\usepackage{graphicx}

\makeatletter
\numberwithin{equation}{section}
\numberwithin{figure}{section}
\theoremstyle{plain}

  \theoremstyle{plain}

\@ifundefined{date}{}{\date{}}

\usepackage{amsthm}\usepackage{amsfonts}\usepackage{amscd}\usepackage[all]{xy}\usepackage[english]{babel}

{\theoremstyle{plain}
\newtheorem{theorem}{Theorem}[section]
\newtheorem{proposition}{Proposition}[section]
\newtheorem{lemma}{Lemma}[section]

\newtheorem{remark}{Remark}[section]
\newtheorem{corollary}{Corollary}[section]
}
\newcommand{\nc}{\newcommand}
\nc {\hh}{\check{h}}
\nc {\DD}{\mathcal{D}}
\nc {\CC}{\mathbb{C}}
\nc {\Pp}{\mathbb{P}}
\nc {\Ss}{\mathcal{S}}
\nc {\Pn}{\mathbb{P}^{n}}
\nc {\PP}{\mathbb{P}^{2}}
\nc {\Pd}{ \check{\mathbb{P}}^{2}}
\nc {\WW}{\mathcal{W}}
\nc {\Sym}{\mathrm{Sym}}
\nc {\OO}{\mathcal{O}}
\nc {\UU}{\mathcal{U}}
\nc {\EE}{\mathcal{E}}
\nc {\MM}{\mathcal{M}}
\nc {\KK}{\mathcal{K}}
\nc {\PW}{\mathcal{P}}
\nc {\NW}{\mathcal{N}_{\WW}}
\nc {\FF}{\mathcal{F}}
\nc {\GG}{\mathcal{G}}
\nc {\ZZ}{\mathcal{Z}}
\nc {\LL}{\mathcal{L}}
\nc {\NN}{\mathcal{N}}
\nc {\RR}{\mathcal{R}}
\nc {\Ww}{\mathbb{W}}
\nc {\QQ}{\mathbb{Q}}
\nc {\II}{\mathcal{I}}

\usepackage{babel}
\providecommand{\theoremname}{Theorem}

\usepackage{babel}
\providecommand{\theoremname}{Theorem}

\usepackage{babel}

\providecommand{\theoremname}{Theorem}

\usepackage{babel}

\providecommand{\theoremname}{Theorem}

\usepackage{babel}

\makeatother

\usepackage{babel}
  \providecommand{\propositionname}{Proposition}
\providecommand{\theoremname}{Theorem}

\begin{document}

\title{The polar family of webs and foliations}

\author{M. Falla Luza}
\address{Departamento de Análise, IM, Universidade Federal Fluminense, Rua Mário Santos Braga s/n - Niterói, 24.020-140 RJ, Brazil.}
\email{maycolfl@impa.br}

\author{R. Rosas Bazan}
\address{Departamento de Matemática, Pontificia Universidad Católica del Perú,
Av. Universitaria 1801, San Miguel, Lima, Perú}
\address{IMCA, Calle los Biólogos 245, La Molina, Lima, Perú}
\email{rudy.rosas@pucp.pe}

\thanks{2000 Mathematics Subject Classification: 37F75}

\thanks{Key words: Holomorphic Web, Polar Curve, Polar Family}

\thanks{The first author was partially supported by Mathamsud/CAPES. He specially thanks the invitation of PUCP at Lima - Peru in July 2012. The second author was partially supported by Mathamsud/CONCYTEC. He thanks the invitation of UFF at Rio de Janeiro - Brazil in March 2013.}

\maketitle
\begin{abstract}
We define the polar curves and the polar family associated to a projective web and obtain some results about the geometry of the generic element of this family. We also deal with the particular case of foliations and prove the constancy of the topological embedded type of the generic polar.
\end{abstract}

\section{Introduction}

Polar subvarieties of foliations have been extensively studied by many authors,
cf. \cite{Mol,Mol-1,Soa}. For example, in \cite{Mol-1,Soa} the authors
bound the degree of an algebraic subvariety which is invariant by
a foliation by using the polar classes of the foliation. Most recently,
in \cite{Falla-Fass} the authors obtained some more general results
for projective webs of any dimension, again by using polar classes.

In \cite{Mol} R. Mol made an study of the properties of the generic
polar curve of a foliation on $\PP$. He proves for instance that
the generic polar curve of a holomorphic foliation is irreducible
and computes its genus.

The aim of this work is to study polar curves of projective webs.
Loosely speaking, a $k$-web is locally given by $k$ foliations
on the complement of a Zariski closed set. We obtain three main results,
the first one is a characterization of the webs having a reducible
generic polar, we prove that this is the case only when the web is
decomposable or its degree is zero, see Theorem \ref{irreducibilidad}.
Our second result says that the polar family characterizes the web,
in other words, two webs having the same polar family are equal, see
Theorem \ref{Familia-polar}.

Finally, in the last section, we return to study polar curves of foliations,
and prove that the topological embbeded type of the generic polar
is constant, see Theorem 4.1. This theorem says that the family of
generic polars is a kind of globally equisingular family. This fact
is well known in the local case (see \cite{LR}) and has been used in
the study of curves and foliation singularities (see for example \cite{Cor}, \cite{M}, \cite{Rou}).

We would like to thank Rogério Mol  and Jorge Vitório Pereira for many valuable conversations and sugestions.

\section{Definitions and some properties}

Roughly speaking, web geometry is the study of invariants for finite
families of foliations. A $k$-web of codimension one on $\Pn$ is
given by an open covering $\mathcal{U}=\{U_{i}\}$ of $\Pn$ and $k$-symmetric
$1$-forms $\omega_{i}\in\Sym^{k}\Omega_{\Pn}^{1}(U_{i})$ subject
to the conditions: 
\begin{enumerate}
\item For each non-empty intersection $U_{i}\cap U_{j}$ there exists a
non-vanishing function $g_{ij}\in\OO_{U_{i}\cap U_{j}}^{*}$ such
that $\omega_{i}=g_{ij}\omega_{j}$. 
\item For every $i$ the zero set of $\omega_{i}$ has codimension at least
two. 
\item For every $i$ and a generic $x\in U_{i}$, the germ of $\omega_{i}$
at $x$ seen as homogeneous polynomial of degree $k$ in the ring
$\mathcal{O}_{x}[dx_{1},...,dx_{n}]$ is square--free. 
\item For every $i$ and a generic $x\in U_{i}$, the germ of $\omega_{i}$
at $x$ is a product of $k$ various $1$-forms $\beta_{1},\ldots,\beta_{k}$,
where each $\beta_{i}$ is integrable. 
\end{enumerate}
The $k$-symmetric $1$-forms $\{\omega_{i}\}$ patch together to
form a global section $\omega=\{\omega_{i}\}\in H^{0}(\PP,Sym^{k}\Omega_{\Pn}^{1}\otimes\mathcal{L})$
where $\mathcal{L}$ is the line bundle over $\Pn$ determined by
the cocycle $\{g_{ij}\}$. The \textbf{singular set} of $\WW$, denoted
by ${\rm {Sing}}(\WW)$, is the zero set of the twisted $k$-symmetric
$1$-form $\omega$. The \textbf{degree} of $\WW$, denoted by $\deg(\WW)$,
is geometrically defined as the degree of the tangency locus between
$\WW$ and a generic $\mathbb{P}^{1}$ linearly embedded in $\Pn$.
If $i:\mathbb{P}^{1}\hookrightarrow\Pn$ is the inclusion then the
degree of $\WW$ is the degree of the zero divisor of the twisted
$k$-symmetric $1$-form $i^{*}\omega\in H^{0}(\mathbb{P}^{1},Sym^{k}\Omega_{\mathbb{P}^{1}}^{1}\otimes\mathcal{L}|_{\mathbb{P}^{1}})$.
Since $\Omega_{\mathbb{P}^{1}}^{1}=\mathcal{O}_{\mathbb{P}^{1}}(-2)$
it follows that $\mathcal{L}=\mathcal{O}_{\mathbb{P}^{n}}(\deg(\WW)+2k)$.
We refer to \cite[Section 1.3]{PiPer} for more details and examples.

We say that $x\in\PP$ is a $\textbf{smooth point}$ of $\WW$, for
short $x\in\WW_{sm}$, if $x\notin{\rm {Sing}}(\WW)$ and the germ
of $\omega$ at $x$ satisfies the conditions described in $(3)$
and $(4)$ above. For any smooth point $x$ of $\WW$ we have $k$
distinct (not necessarily in general position) linearly embedded subspaces
of dimension $n-1$ passing through $x$. Each one of these subspaces
will be called hyperplane tangent to $\WW$ at $x$ and denoted by
$T_{x}^{1}\WW,...,T_{x}^{k}\WW$.

Consider the manifold $M=\mathbb{P}(T^{*}\Pn)$ which can be identified
with the incidence variety of points and hyperplanes in $\Pn$, thus
we have two natural projections $\pi$ and $\check{\pi}$ to $\Pn$
and $\check{\mathbb{P}}^{n}$ respectively. To any $k$-web $\WW$
we can associate the subvariety $S_{\WW}$ of $M$ defined as 
\[
S_{\WW}=\overline{\{(x,H)\in M:x\in\WW_{sm}\hspace{0.1cm}and\hspace{0.1cm}\exists\hspace{0.1cm}1\leq i\leq k,\hspace{0.1cm}H=T_{x}^{i}\WW\hspace{0.1cm}\}}.
\]

The $\textbf{discriminant}$ of $\WW$, denoted by $\Delta(\WW)$, is the projection of the ramification
variety of $\pi|_{S_{\WW}}$. Given a point $p\in\Pn$, we set $\LL_{p}$
the family of hyperplanes passing through $p$. This set can be thought
as an hyperplane in $\check{\mathbb{P}}^{n}$ and then one can take
$\mathcal{H}_{p}=\check{\pi}^{-1}(\LL_{p})\subseteq M$. We define
\textbf{the polar set of $\WW$ with center on $p$}, denoted by $P_{p}^{\WW}$,
as 
\[
P_{p}^{\WW}=tang(\WW,\LL_{p}):=\pi(S_{\WW}\cap\mathcal{H}_{p}).
\]
We also denote by $\widetilde{P_{p}^{\WW}}=S_{\WW}\cap\mathcal{H}_{p}$
which is a subvariety in $S_{\WW}$ projecting onto $P_{p}^{\WW}$.\\
 In this work we shall focus on the case of $\PP$. Observe that in
this case $\LL_{p}$ is the radial foliation with singularity at $p$
and $P_{p}^{\WW}$ is the curve of tangencies between this foliation
and the web.\\
 In what follows we list some properties of the polar curve which
are already known for foliations, cf. \cite{Mol}.

\begin{proposition} Let $\WW$ be a $k$ web of degree $d$ on $\PP$,
then 
\begin{enumerate}
\item The polar curve with center in $p$ is the whole $\PP$ if and only
if $\WW$ can be written as the product $\WW=\LL_{p}\boxtimes\WW'$,
for some $(k-1)$ web $\WW'$.\\

\item If $P_{p}^{\WW}\subsetneq\PP$, then $degP_{p}^{\WW}=d+k$. 
\end{enumerate}
\end{proposition} \begin{proof} \hfill{} 
\begin{enumerate}
\item Let us assume that $P_{p}^{\WW}=\PP$. Since $\pi|_{\mathcal{H}_{p}}:\mathcal{H}_{p}\rightarrow\PP$
is birrational one has $\mathcal{H}_{p}\cap S_{\WW}=\mathcal{H}_{p}$
and then $\mathcal{H}_{p}$ is an irreducible component of $S_{\WW}$.
The converse is clear. 
\item See \cite[Proposition 2.1]{Falla-Fass} where is proved the same in
the case of $\Pn$. 
\end{enumerate}
\end{proof}

\begin{proposition} Let $\WW_{1}=[\omega_{1}]$ and $\WW_{2}=[\omega_{2}]$
be $k$-webs of degree $d$ and $p \in \PP$, then $P_{p}^{\WW_{1}}=P_{p}^{\WW_{2}}$
if and only if the web given by $\omega_{2}-\omega_{1}$ can be written as the product $\LL_{p}\boxtimes\WW'$ for some $(k-1)$ web $\WW'$. 
\end{proposition} \begin{proof} Let us
suppose that $P_{p}^{\WW_{1}}=P_{p}^{\WW_{2}}$ and denote by $S_{\WW_{i}}=\{F_{i}=0\}$.
By hypothesis we have $S_{\WW_{1}}\cap\mathcal{H}_{p}=S_{\WW_{2}}\cap\mathcal{H}_{p}$
so one can assume that $(F_{1}-F_{2})|_{\mathcal{H}_{p}}\equiv0$.
Therefore $\mathcal{H}_{p}$ is a component of the surface defined
by $F_{1}-F_{2}$. The converse part is clear. \end{proof}

Now we denote by $C(r)=\mathbb{P}H^{0}(\PP,\OO_{\PP}(r))$ the projective
space of curves of degree $r$ and, for a $k$-web $\WW$, by $\mathcal{P}=\{p_{1},\ldots,p_{m}\}$
the set of centers of radial foliations tangent to $\WW$ (which could
be empty). The \textbf{polar family} of $\WW$ is defined as the subvariety
\[
\RR(\WW):=\overline{\{P_{p}^{\WW}:p\in\PP-\mathcal{P}\}}\subseteq C(d+k).
\]
In other words, $\RR(\WW)$ is the closure of the image of the rational
map 
\[
\mathcal{G}:\PP-\mathcal{P}\rightarrow C(d+k),\hspace{0.3cm}\mathcal{G}(p)=P_{p}^{\WW}.
\]

\begin{remark} In the case $\WW=\LL_{p}$ we have $\RR(\WW)=$ \{lines
through $p$\}, which is a line in $C(1)=\Pd$. \end{remark}

\begin{proposition} If $\WW$ is a $k$-web of degree $d$ different
from $\LL_{p}$ then dim$\RR(\WW)$ $=2$ and deg$\RR(\WW)$ $=k^{2}$.
\end{proposition} 
\begin{proof} 
If $\RR(\WW)$ is an irreducible
curve, then for a generic $p\in\PP$ the inverse image $\mathcal{G}^{-1}(P_{p}^{\WW})$
is a curve. Therefore the point $p$ belongs to $P_{q}^{\WW}$ for
infinitely many points $q\in\mathcal{G}^{-1}(P_{p}^{\WW})$. The only
possibility is $\mathcal{G}^{-1}(P_{p}^{\WW})$ to be a line invariant
by $\WW$. Since the only foliations by lines on $\PP$ are the foliations of degree zero (lines through a point), $\mathcal{G}$ defines a radial foliation $\LL_{p}$ tangent to $\WW$ for some point $p\in\PP$. Fix now a generic point
$q\in\PP$; then it is clear that the leaf of $\LL_{p}$ passing
through $q$ is contained in $P_{q}^{\WW}$. If the equality does
not hold we can choose a regular point $z\in P_{q}^{\WW}$ which is
out of the leaf, but since all the points of this leaf has the same
polar curve than $q$ this is not possible. We have just proved that
the polar centered in a generic point $q$ is the line passing by
$p$ and $q$. This shows that $\WW=\LL_{p}$.\\
 To find the degree we need just to do the intersection of $\RR(\WW)$
with a generic plane of codimension $2$ in $C(d+k)$, which corresponds to the curves passing through two points. Then we observe that for
generic $p_{1},p_{2}\in\PP$, $\{p_{1},p_{2}\}\subseteq P_{x}^{\WW}$
if and only if $x\in T_{p_{1}}\WW\cap T_{p_{2}}\WW$. We conclude by noting that the cardinal of $T_{p_1}\WW \cap T_{p_2}\WW$ is $k^2$.
\end{proof}

\begin{remark} If we set $B(R(\WW))$ the set of points which belong
to every polar curve, then is easy to see that $B(R(\WW))\subseteq Sing(\WW)$.
\end{remark}

We conclude the section with an important property of the generic
polar curve. The main ingredient for this result is the extended Bertini's
first theorem, cf. \cite[Theorem 4.1]{Kleiman}.\\
 \begin{theorem}\label{singularidades-polar} For a generic point
$p\in\PP$ we have 
\begin{enumerate}
\item $Sing(P_{p}^{\WW})\subseteq\Delta(\WW)\cup\{p\}$. 
\item The are $k$ smooth and transversal branches of $P_{p}^{\WW}$ passing
through $p$, tangent to the $k$ directions of $T_{p}\WW$. 
\end{enumerate}
\end{theorem} \begin{proof} Let us consider the linear system on
$S_{\WW}$ formed by the curves $\widetilde{P_{p}^{\WW}}=\check{\pi}|_{S_{\WW}}^{-1}(\check{p})$.
It is clear that this system has no fixed component, therefore we
can apply Bertini's first theorem to conclude that the generic element
is smooth outside of the base locus. Take now a generic $P_{p}^{\WW}=\pi(\widetilde{P_{p}^{\WW}})$
and $x\in sing(P_{\WW}^{p})$. We have the following possibilities: 
\begin{itemize}
\item x comes from a ramification point of $\pi|_{S_{\WW}}$ and so is a
point of $\Delta(\WW)$. 
\item x comes from a singularity of $\widetilde{P_{p}^{\WW}}$ and so comes
from a base point. 
\item There are two regular points $z_{1},z_{2}\in\widetilde{P_{p}^{\WW}}$,
out of the ramification of $\pi|_{S_{\WW}}$ which project to $x$.
If $x\neq p$, then $\WW$ decomposes in a neighborhood of $x$ and
$\LL_{p}$ is tangent in $x$ just to one of the foliations of this
decomposition. Then there is just one point in $\widetilde{P_{p}^{\WW}}$
over $x$, contradiction. Therefore $x=p$. 
\end{itemize}
This concludes the proof of the first assertion.\\
 In order to prove the second part, take local coordinates such that
$p=(0,0)$ and $\WW$ is given near to $p$ by the regular $1$-forms
$A_{1}dx+B_{1}dy,\ldots,A_{k}dx+B_{k}dy$. Then an equation for $P_{p}^{\WW}$
is $(xA_{1}+yB_{1})\ldots(xA_{k}+yB_{k})=0$. Since 
\[
xA_{j}+yB_{j}=xA_{j}(0)+yB_{j}(0)+h.o.t
\]
we conclude that $P_{p}^{\WW}$ has $k$ smooth transversal branches
through $p$. \end{proof}

\begin{remark} Observe that in the proof of (2) we only use $p\notin\Delta(\WW)$,
in particular, for a foliation $\FF$, the polar curve $P_{p}^{\FF}$
is smooth at $p$ when $p$ is a regular point of $\FF$. \end{remark}

\section{Irreducibility of the generic polar}

The goal of this section is to classify the webs whose generic polar
curve is decomposable. We begin with an easy observation.

\begin{remark} If one considers a decomposable web $\WW=\WW_{1}\boxtimes\ldots\boxtimes\WW_{r}$,
then $P_{p}^{\WW}$ has $P_{p}^{\WW_{1}},\ldots,P_{p}^{\WW_{r}}$
as components and therefore is decomposable. \end{remark}

\begin{theorem}\label{irreducibilidad} Let $\WW$ be an irreducible
$k$-web with $k\geq2$. Then $P_{p}^{\WW}$ is decomposable for generic
$p\in\PP$ if and only if deg $\WW$ $=0$. \end{theorem}

\begin{proof} By hypothesis the surface $S_{\WW}$ is irreducible,
then we apply the second extended Bertini's theorem (cf. \cite[Theorem 5.3]{Kleiman})
to conclude that the generic curve $\widetilde{P_{p}^{\WW}}$ is reducible
if and only if the system is a composite with a pencil. So if we suppose
that the generic polar is decomposable then the image of $\check{\pi}|_{S_{\WW}}$
is a curve and this imply deg $\WW$ $=0$ (see \cite[Proposition 1.4.2]{PiPer}).
The converse part is clear. \end{proof}

\begin{corollary} For a $k$-web $\WW$ the generic polar is decomposable
if and only if $\WW$ decomposable or deg $\WW$ $=0$ and $k\geq2$.
\end{corollary}

An interesting fact about the polar curves is that they determine
the web as we will see.

\begin{lemma}\label{lemma-fundamental} For $i=1,2$ let $\WW_{i}$
be a $k_{i}$-web of degree $d_{i}$ and suppose $k_{1}\geq2$. If
$\RR(\WW_{1})=\RR(\WW_{2})$ then $P_{p}^{\WW_{1}}=P_{p}^{\WW_{2}}$
for every $p\in\PP$ and $(k_{1},d_{1})=(k_{2},d_{2})$. \end{lemma}
\begin{proof} We recall that $\RR(\WW_{i})$ is the image of the
rational map $\mathcal{G}_{i}:\PP\dashrightarrow C(d_{i}+k_{i})$.
Take now $P_{p}^{\WW_{1}}\in\RR(\WW_{1})=\RR(\WW_{2})$ generic for
$\WW_{1}$ and $\WW_{2}$ in the sense of Theorem \ref{singularidades-polar}
which is outside of $\mathcal{G}_{1}(\Delta(\WW_{2}))$. Then $P_{p}^{\WW_{1}}=P_{q}^{\WW_{2}}$
for some point $q$. As $P_{q}^{\WW_{2}}$ has $k_{1}$ branches through
$p\notin\Delta(\WW_{2})$, Theorem \ref{singularidades-polar} implies
that $p=q$ and $k_{1}=k_{2}$. To conclude we only need to point
that $k_{1}+d_{1}=k_{2}+d_{2}$. \end{proof}

\begin{theorem}\label{Familia-polar} If $\WW_{1}$ and $\WW_{2}$
are $k$-webs of degree $d$, with $k\geq2$, such that $\RR(\WW_{1})=\RR(\WW_{2})$,
then $\WW_{1}=\WW_{2}$. \end{theorem}

\begin{proof} Let $p\in\PP$ a generic point. By the previous lemma
$P_{p}^{\WW_{1}}=P_{p}^{\WW_{2}}$, so the branches by $p$ are the
same and by the second part of Theorem \ref{singularidades-polar}
we have $T_{p}\WW_{1}=T_{p}\WW_{2}$. \end{proof}

\begin{remark} The proof of Lemma \ref{lemma-fundamental} does not
hold for foliations. Actually, in \cite{Cam-Oliv} we can find different
foliations of degree one with the same polar family. However, in the
same work the authors showed that Theorem \ref{Familia-polar} remains
true for foliations of degree $d\geq2$. \end{remark} 

\section{Polar curves of foliations}

Now we focus in the case of foliations. Let $\FF$ be a foliation
on $\PP$ of degree $d \geq 1$, in this case the polar family $\RR(\FF)$ is a net. We recall from \cite{Per} the definition
of the inflexion divisor of $\FF$, denoted by $\mathcal{E}(\FF)$,
as the curve formed (out of the singularities) by the inflexion points
of the leaves of $\FF$. If the foliation is given locally by the
vector field $X=A\frac{\partial}{\partial x}+B\frac{\partial}{\partial y}$,
then 
\[
\mathcal{E}(\FF)=\left\{ B^{2}A_{y}+ABA_{x}-A^{2}B_{x}-ABB_{y}=0\right\} .
\]

\begin{lemma} For every point $p\in\PP$ we have $Sing(P_{p}^{\FF})\subseteq\mathcal{E}(\FF)$.
\end{lemma}

\begin{proof} It is a straightforward computation to show that a
point $q\in P_{p}^{\FF}$ which does not belong to $\mathcal{E}(\FF)$
cannot be a singularity of $P_{p}^{\FF}$.\end{proof}

We say that a singular point $q\in Sing(\FF)$ is \textbf{quasi-radial}
if after doing a blow up at $q$ the exceptional divisor is not invariant
by the strict transform of $\FF$. 
\begin{proposition}
For generic $p\in\PP$ the polar curve $P_{p}^{\FF}$ has the following
properties:
\begin{enumerate}
\item If $q\in Sing(\mathcal{F})$ is not quasi-radial, the tangent cone of
$P_{p}^{\FF}$ at $q$ does not contain the line $\overline{pq}$,
\item If $q\in Sing(\mathcal{F})$ is quasi-radial, the tangent cone of $P_{p}^{\FF}$at
$q$ contains the line $\overline{pq}$ exactly one time.
\end{enumerate}
\end{proposition}
\begin{proof} Take local coordinates such that $q=(0,0)$ and $\FF$
is given locally by the vector field $X=A\frac{\partial}{\partial x}+B\frac{\partial}{\partial y}=(A_{k}+A_{k+1}+\ldots)\frac{\partial}{\partial x}+(B_{k}+B_{k+1}+\ldots)\frac{\partial}{\partial y}$,
where $A_{j}$ and $B_{j}$ are homogeneous polynomial of degree $j$.
Take now a generic point $p=(a,b)$, then an equation for $P_{p}^{\FF}$
is 
\[
(x-a)B-(y-b)A=bA_{k}-aB_{k}+h.o.t=0.
\]

Thus the tangent cone of $P_{p}^{\FF}$at $q$ contains the line $\overline{pq}$
if and only if $bA_{k}(a,b)-aB_{k}(a,b)=0$. If that holds for $p$
generic, it is equivalent to have $yA_{k}(x,y)-xB_{k}(x,y)\equiv0$,
which says that $q$ is a quasi-radial singularity of $\FF$. In this
case one writes $A_{k}=xP$ and $B_{k}=yP$ for some polynomial $P$
and then $bA_{k}-aB_{k}=(bx-ay)P$. Because $p$ is generic we can
assume that $bx-ay$ is not a factor of $P$. 

\end{proof}

The precedent proposition says that, for $p$ generic, a quasi-radial
singularity of $\FF$ gives exactly one intersection between $\check{P}_{p}^{\FF}$
and $\check{\mathcal{L}_{p}}$ (the duals of $P_{p}^{\FF}$ and $\mathcal{L}_{p}$ respectively) and a non quasi-radial singularity gives
no intersections between $\check{P}_{p}^{\FF}$ and $\check{\mathcal{L}_{p}}$.
On the other hand, the point $p$ give exactly one intersection between
$\check{P}_{p}^{\FF}$ and $\check{\mathcal{L}_{p}}$, as the following
lemma shows.

\begin{lemma}\label{inflexion-estatica} Let $p$ be a regular point
of $\FF$. Then $p$ is an inflexion point of $P_{p}^{\FF}$ if and
only if $p\in\EE(\FF)$. \end{lemma}

\begin{proof} Take local coordinates such that $p=(0,0)$ and $\FF$
is given by the vector field $\frac{\partial}{\partial x}+B\frac{\partial}{\partial y}$
with $B(0,0)=0$. Therefore $P_{p}^{\FF}=\{y-xB=0\}$ has inflexion
at $p$ if and only if the Hessian of $y-xB$ vanishes on the tangent
direction of $P_{p}^{\FF}$ at $p$. By a straightforward computation
this is equivalent to have $\frac{\partial B}{\partial x}(0,0)=0$.
On the other hand, the static curve $\EE(\FF)$ pass through $p$
if and only if $\frac{\partial B}{\partial x}(0,0)=0$. \end{proof}

Putting all this together as in \cite[Proposition 4.1]{Mol} we obtain
a bound for the number of quasi-radial singularities:

\[
\#Sing_{QR}(\FF)\leq deg(\check{P}_{p}^{\FF})-1.
\]

We recall that R. Mol obtains this inequality for quasi-radial singularities
of multiplicity one.

Now we state the main result of this section, which establish the
invariance of the topological type of the generic polar.

\begin{theorem}\label{polar-topologia-constante} For a foliation
$\mathcal{F}$, the generic polar has constant topological embedded
type in $\PP$, that is: if $P_{p_{1}}^{\FF}$ and $P_{p_{2}}^{\FF}$
are generic polar curves, then there exists a homeomorphism $H:\PP\rightarrow\PP$such
that $H(P_{p_{1}}^{\FF})=P_{p_{2}}^{\FF}$ . In particular, the genus
of the generic polar is constant.

\end{theorem}

\begin{proof} It is sufficient to prove that the generic polar has
locally constant topological embedded type in $\PP$. The polar net
defines a local linear system on each singularity $q_{1},\ldots,q_{l}$
of $\FF$. For generic $p\in\PP$, the family $(P_{p}^{\FF},q_{j})$
is equisingular $(j=1,\ldots,l)$ and each curve is reduced (see \cite{Casas}
chapter 7). Moreover, for each $j=1,\ldots,l$, there exists a finite
sequence of blow-ups $\pi_{j}$ at $q_{j}$with the following properties: 
\begin{enumerate}
\item The sequence of blow-ups $\pi_{j}$ desingularizates the generic curve
$(P_{p}^{\FF},q_{j})$. 
\item There is a list $D_{1}^{j},\ldots,D_{N_{j}}^{j}$of components (not
necessarily different) of the exceptional divisor of $\pi_{j}$ such
that the strict transform of the generic curve $(P_{p}^{\FF},q_{j})$
has exactly $N_{j}$ branches $\gamma_{1}(p),\ldots,\gamma_{N_{j}}(p)$
passing (transversely) through $D_{1}^{j},\ldots,D_{N_{j}}^{j}$ respectively. 
\end{enumerate}
Let $\zeta_{k}^{j}(p)$ be the intersection between $D_{k}^{j}$ and
$\gamma_{k}^{j}(p)$. Denote by $M$ the surface resulting by the
composition $\pi$ of all the blow-ups involved in $\pi_{1},\ldots,\pi_{l}$,
by $E$ the total exceptional divisor and by $\overline{P_{p}}\subseteq M$
the strict transform of $P_{p}^{\FF}$. Observe that $\overline{P_{p}}$
is smooth and intersects the exceptional divisor at the points $\{\zeta_{k}^{j}(p):j=1,\ldots,l;\hspace{1em}k=1,\ldots,N_{j}\}$.
\begin{lemma}\label{funciones definidoras} For any $z\in M$, there
is a neighborhood $U$ of $z$ and an analytic family $\{f_{p}^{U}\}$of
functions on $U$ such that $\overline{P_{p}}\cap U=\{f_{p}^{U}=0\}$
for generic $p$. \end{lemma} \begin{proof} The proposition is obvious
if $z\notin E$. Let us take $z\in E$ and suppose first that $z$
is not a corner point. In a neighborhood of $\pi(z)\in\PP$ the polar
family is defined by an analytic family of equations $\{h_{p}\}$.
Let $g=0$ be a local reduced equation of $E$ at $z$. Then, since
the resolution of the family $P_{p}^{\FF}$ ($p$ generic) is achieved
with the same sequence of blow-ups, there exists $n\in\mathbb{N}$
such that, in a neighborhood of $z$, the curve $\overline{P_{p}}$
is defined by the equation $f_{p}=\frac{h_{p}\circ\pi}{g^{n}}=0$
for generic $p$. In the case of a corner, the proof is similar. \end{proof}
Fix now a generic $p_{0}\in\PP$ and consider the curve $\overline{P}=\overline{P_{p_{0}}}$.
According with the tubular neighborhood theorem, the curve $\overline{P}$
admits a neighborhood $T$ fibered by smooth two dimensional disks
transverse to $\overline{P}$. This fibration can be taken such that
the intersection of $E$ with $T$ is invariant by the fibration,
i.e. $E\cap T$ is a finite union of disks.

\begin{center}
\includegraphics[width=6cm]{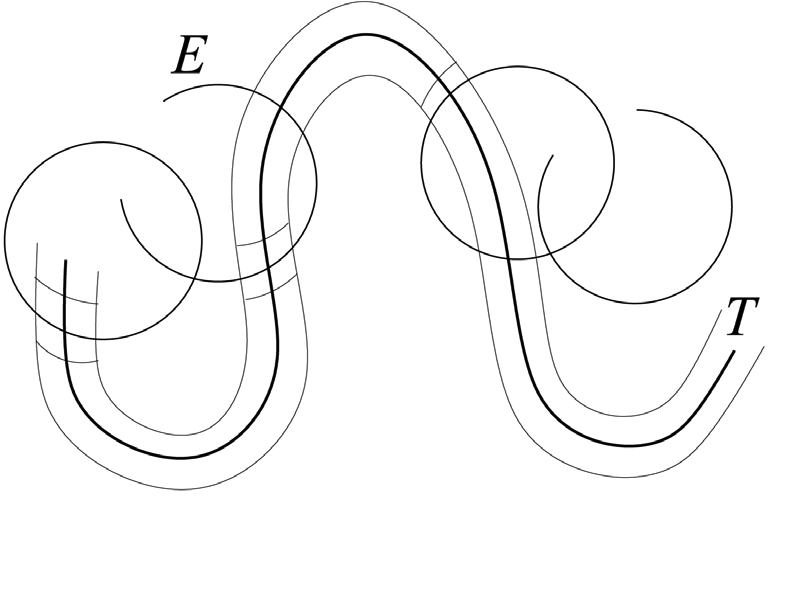}
\par\end{center}

From Lemma \ref{funciones definidoras} it is easy to prove that for
$p$ close enough to $p_{0}$, the curve $\overline{P_{p}}$ is contained
in $T$. We will use the following elementary lemma. \begin{lemma}\label{submersion-grafico}
Let $f:\overline{\mathbb{D}}\times\overline{\mathbb{D}}\rightarrow\mathbb{C}$
be a $C^{\infty}$submersion such that $f^{-1}(0)=\overline{\mathbb{D}}\times\left\{ 0\right\} $.
If $h:\overline{\mathbb{D}}\times\overline{\mathbb{D}}\rightarrow\mathbb{C}$
is close enough to $f$ in the $C^{1}$ topology, then $h^{-1}(0)$
is the graph of a function $\psi:\overline{\mathbb{D}}\rightarrow\mathbb{D}$.
\end{lemma}

Let $f_{p_{0}}^{U_{1}},\ldots,f_{p_{0}}^{U_{r}}$ be a finite set
of defining functions of $\overline{P}$ given by Lemma \ref{funciones definidoras},
$\overline{P}\subset U_{1}\cup\ldots\cup U_{r}$. Since the generic
polar is reduced and $\overline{P}$ is smooth, the functions $f_{p_{0}}^{U_{1}},\ldots,f_{p_{0}}^{U_{r}}$
are holomorphic submersions at points of $\overline{P}$. Thus, by
reducing $T$ and the open sets $U_{1},\ldots U_{r}$ if necessary
we may assume that 
\begin{enumerate}
\item The functions $f_{p_{0}}^{U_{1}},\ldots,f_{p_{0}}^{U_{r}}$ are submersions 
\item Any $z\in\overline{P}$ has a neighborhood $\Omega$ in $\overline{P}$ with the following property: If $D_\zeta$ denote the disc of the fibration on $T$ passing through $\zeta\in\overline{P}$, the closure of the set $\displaystyle{\bigcup_{\zeta\in\Omega}D_\zeta}$  is contained in some $U_{i}$. 
\end{enumerate}
Given $z\in\overline{P}$, take a neighborhood $\Omega$ of $z$ in
$\overline{P}$ as above an such that $\overline{\Omega}$ is diffeomorphic
to $\overline{\mathbb{D}}$. Let $T_{\Omega}$ be the restriction
to $\Omega$ of the fibration on $T$. Clearly $\overline{T_{\Omega}}$
admits $C^{\infty}$coordinates $(x,y)\in\overline{\mathbb{D}}\times\overline{\mathbb{D}}$
such that $\overline{P}$ is given by $\overline{\mathbb{D}}\times\{0\}$
and fibers are given by the sets $\{x\}\times\overline{\mathbb{D}}$.
Take some $f_{p_{0}}^{U_{i}}$ such that $\overline{T_{\Omega}}$
is contained in $U_{i}$. Then $f_{p_{0}}^{U_{i}}$ is a submersion
on $\overline{T_{\Omega}}$ and by Lemma \ref{submersion-grafico}
we see that for $p$ generic and close to $p_{0}$ the curve $\overline{P_{p}}$
intersects once each fiber of $T_{\Omega}$. By compacity we conclude
that for $p$ generic and close to $p_{0}$ the curve $\overline{P_{p}}$
is contained in $T$ and intersects once each fiber of $T$. Then
it is easy to construct a homeomorphism $H:\overline{T}\rightarrow\overline{T}$
with the following properties: 
\begin{enumerate}
\item $H$ fix the fibers 
\item $H$ maps $\overline{P_{p}}$ onto $\overline{P}$ 
\item $H$ is the identity when restricted to $\partial T$. 
\end{enumerate}
Clearly $H$ extends as a homeomorphism $H:M\rightarrow M$ by making
$H(z)=z$ for $z\notin\overline{T}$ and it is easy to see that $H(E)=E$.
Then $H$ induces a homeomorphism $h:\mathbb{P}^{2}\rightarrow\mathbb{P}^{2}$
such that $h(P_{p}^{\mathcal{F}})=P_{p_{0}}^{\mathcal{F}}$.

\end{proof}

\end{document}